\documentclass[numbook,envcountsame,smallcondensed,draft]{amsart}

\usepackage{amssymb,amsmath,amsfonts,cite,color}

\DeclareMathOperator{\con}{con}

\DeclareMathOperator{\simple}{sim}

\DeclareMathOperator{\mul}{mul}

\DeclareMathOperator{\End}{End}

\DeclareMathOperator{\occ}{occ}

\DeclareMathOperator{\var}{var}

\newtheorem{theorem}{Theorem}[section]

\newtheorem{proposition}[theorem]{Proposition}

\newtheorem{lemma}[theorem]{Lemma}

\newtheorem{corollary}[theorem]{Corollary}

\newtheorem{remark}[theorem]{Remark}

\theoremstyle{definition}

\newtheorem{example}[theorem]{Example}

\numberwithin{equation}{section}

\makeatletter

\renewcommand*\subjclass[2][2010]{\def\@subjclass{#2}\@ifundefined{subjclassname@#1}{\ClassWarning{\@classname}{Unknown edition (#1) of Mathematics Subject Classification; using '2010'.}}{\@xp\let\@xp\subjclassname\csname subjclassname@#1\endcsname}}

\renewcommand{\subjclassname}{\textup{2010} Mathematics Subject Classification}

\makeatother

\begin{document}

\title[A new example of a limit variety of monoids]{A new example of a limit variety of monoids}

\thanks{The work is supported by the Ministry of Education and Science of the Russian Federation (project 1.6018.2017/8.9) and by Russian Foundation for Basic Research (grant 17-01-00551)}

\author{S.V.Gusev}

\address{Ural Federal University, Institute of Natural Sciences and Mathematics, Lenina 51, 620000 Ekaterinburg, Russia}

\email{sergey.gusb@gmail.com}

\date{}

\begin{abstract}
A variety of universal algebras is called limit if it is non-finitely based but all its proper subvarieties are finitely based. Until recently, only two explicit examples of limit varieties of monoids, constructed by Jackson, were known. Recently Zhang and Luo found the third example of such a variety. In our work, one more example of a limit variety of monoids is given.
\end{abstract}



\maketitle

\section{Introduction and summary}
\label{introduction}

A variety of algebras is called \emph{finitely based} if it has a finite basis of its identities, otherwise, the variety is said to be \emph{non-finitely based}. Much attention is paid to the study of finitely based and non-finitely based varieties of algebras of various types. In particular, the finitely based and non-finitely based varieties of semigroups and monoids have been the subject of an intensive research (see the surveys~\cite{Shevrin-Volkov-85,Volkov-01}).

A variety is called a \emph{limit variety} if it is non-finitely based but every proper subvariety
is finitely based. Limit varieties play an important role because each non-finitely based variety contains some limit subvariety. There are continuum many limit varieties of groups~\cite{Kozhevnikov-12}. When studying varieties of semigroups and monoids, the varieties play an important role that are away from group varieties in a sense. We mainly mean varieties of \emph{aperiodic} monoids, i.e., monoids that have trivial subgroups only. A few explicit examples of limit varieties of monoids are known so far, and all these varieties consist of aperiodic monoids. In~\cite{Jackson-05}, Jackson found the first two examples of such varieties $\mathbf J_1$ and $\mathbf J_2$. Lee established that only $\mathbf J_1$ and $\mathbf J_2$ are limit varieties within several classes of monoid varieties~\cite{Lee-09,Lee-12}. In~2013, Zhang found a non-finitely based variety $\mathbf L$ of aperiodic monoids that does not contain the varieties $\mathbf J_1$ and $\mathbf J_2$~\cite{Zhang-13} and, therefore, she proved that there exists a limit variety of monoids that differs from $\mathbf J_1$ and $\mathbf J_2$. Just recently, Zhang and Luo identified an explicit example of such variety~\cite{Zhang-Luo-19+}. In this article we exhibit another example of a limit variety of aperiodic monoids. Note that our limit variety is not contained in the variety $\mathbf L$.

In order to formulate the main result of the article, we need some notation. The free monoid over a countably infinite alphabet is denoted by $F^1$. As usual, elements of $F^1$ and the alphabet are called \emph{words} and \emph{letters} respectively. Words and letters are denoted by small Latin letters. However, words unlike letters are written in bold. Expressions like to $\mathbf u\approx\mathbf v$ are used for identities, whereas $\mathbf{u=v}$ means that the words $\mathbf u$ and $\mathbf v$ coincide. As usual, the symbol $\mathbb N$ stands for the set of all natural numbers. For an arbitrary $n\in\mathbb N$, we denote by $S_n$ the full symmetric group on the set $\{1,2,\dots,n\}$. If $\mathbf \pi\in S_n$ then we put
$$
\mathbf w_n[\pi]= x\, z_{1\pi}z_{2\pi}\cdots z_{n\pi}\, x\, \biggl(\,\prod_{i=1}^n t_iz_i\biggr)\ \text{ and }\ \mathbf w_n'[\pi]=x^2\, z_{1\pi}z_{2\pi}\cdots z_{n\pi}\, \biggl(\,\prod_{i=1}^n t_iz_i\biggr).
$$
We fix notation for the following identity system:
$$
\Phi=\{xyx\approx xyx^2,\,x^2y^2\approx y^2x^2,\, xyzxy\approx yxzxy\}.
$$
For an identity system $\Sigma$, we denote by $\var\,\Sigma$ the variety of monoids given by $\Sigma$. Put
$$
\mathbf J=\var\,\{\Phi,\,xyxztx\approx xyxzxtx,\,\mathbf w_n[\pi]\approx \mathbf w_n'[\pi]\mid n\in \mathbb N,\ \pi\in S_n\}.
$$
A variety is called \emph{finitely generated} if it is generated by a finite algebra.

The main result of the paper is the following

\begin{theorem}
\label{main result}
The variety $\mathbf J$ is a finitely generated limit variety of monoids.
\end{theorem}

Recall that a variety is called \emph{Cross} if it is finitely based, finitely generated and small. A non-Cross variety is said to be \emph{almost Cross} if all its proper subvarieties are Cross. As we will see below, Theorem~\ref{main result} implies

\begin{corollary}
\label{J is Cross}
The variety $\mathbf J$ is almost Cross.
\end{corollary}

We note that only a few papers with new explicit examples of almost Cross varieties of aperiodic monoids are known. They are~\cite{Jackson-05,Lee-13,Gusev-Vernikov-18,Jackson-Lee-18,Wismath-86,Zhang-Luo-19+}.

If $\mathbf X$ is a monoid variety then we denote by $\overleftarrow{\mathbf X}$ the variety \emph{dual to} $\mathbf X$, i.e., the variety consisting of monoids antiisomorphic to monoids from $\mathbf X$. A monoid variety $\mathbf X$ is called \emph{self-dual} if $\mathbf X = \overleftarrow{\mathbf X}$. We note that all the limit monoid varieties from~\cite{Jackson-05,Zhang-Luo-19+} are self-dual, while the variety $\mathbf J$ is non-self-dual. So, Theorem~\ref{main result} implies that the variety $\overleftarrow{\mathbf J}$ is a finitely generated limit variety of monoids too.

The article consists of four sections. Section~\ref{preliminaries} contains definitions, notation and auxiliary results. In Section~\ref{the subvariety lattice of J} we describe the subvariety lattice of the variety $\mathbf J$, while Section~\ref{proof of main result} is devoted to the proof of Theorem~\ref{main result}.

\section{Preliminaries}
\label{preliminaries}

Recall that a variety of universal algebras is called \emph{locally finite} if all of its finitely generated members are finite. Varieties with a finite subvariety lattice are called \emph{small}.

\begin{lemma}[\mdseries{\!\!\cite[Lemma~6.1]{Jackson-05}}]
\label{small+LF=>FG}
Every small locally finite variety of algebras is finitely generated.\qed
\end{lemma}

As usual, $\End(F^1)$ denotes the endomorphism monoid of the monoid $F^1$. The following statement is the specialization for monoids of a well-known universal-algebraic fact.

\begin{lemma}
\label{deduction}
The identity $\mathbf u\approx \mathbf v$ holds in the variety of monoids given by an identity system $\Sigma$ if and only if there exists a sequence of words
\begin{equation}
\label{sequence of words}
\mathbf v_0,\mathbf v_1,\ldots, \mathbf v_m
\end{equation}
such that $\mathbf u=\mathbf v_0$, $\mathbf v_m=\mathbf v$ and, for any $0\le i<m$, there are words $\mathbf a_i,\mathbf b_i\in F^1$, an endomorphism $\xi_i\in\End(F^1)$ and an identity $\mathbf s_i\approx \mathbf t_i\in\Sigma$ such that either $\mathbf v_i=\mathbf a_i\xi_i(\mathbf s_i)\mathbf b_i$ and $\mathbf v_{i+1}=\mathbf a_i\xi_i(\mathbf t_i)\mathbf b_i$ or $\mathbf v_i=\mathbf a_i\xi_i(\mathbf t_i)\mathbf b_i$ and $\mathbf v_{i+1}=\mathbf a_i\xi_i(\mathbf s_i)\mathbf b_i$.\qed
\end{lemma}

A letter is called \emph{simple} [\emph{multiple}] \emph{in a word} $\mathbf w$ if it occurs in $\mathbf w$ once [at least twice]. The set of all simple [multiple] letters in a word \textbf w is denoted by $\simple(\mathbf w)$ [respectively $\mul(\mathbf w)$]. The \emph{content} of a word \textbf w, i.e., the set of all letters occurring in $\mathbf w$, is denoted by $\con(\mathbf w)$. We denote the empty word by $\lambda$. The number of occurrences of the letter $x$ in $\mathbf w$ is denoted by $\occ_x(\mathbf w)$. For a word \textbf w and letters $x_1,x_2,\dots,x_k\in \con(\mathbf w)$, let $\mathbf w(x_1,x_2,\dots,x_k)$ be the word obtained from $\mathbf w$ by deleting from $\mathbf w$ all letters except $x_1,x_2,\dots,x_k$. 

Let $\mathbf w$ be a word and $\simple(\mathbf w)=\{t_1,t_2,\dots, t_m\}$. We can assume without loss of generality that $\mathbf w(t_1,t_2,\dots, t_m)=t_1t_2\cdots t_m$. Then $\mathbf w = t_0\mathbf w_0 t_1 \mathbf w_1 \cdots t_m \mathbf w_m$ where $\mathbf w_0,\mathbf w_1,\dots,\mathbf w_m$ are possibly empty words and $t_0=\lambda$. The words $\mathbf w_0$, $\mathbf w_1$, \dots, $\mathbf w_m$ are called \emph{blocks} of a word $\bf w$, while $t_0,t_1,\dots,t_m$ are said to be \emph{dividers} of $\mathbf w$. The representation of the word \textbf w as a product of alternating dividers and blocks, starting with the divider $t_0$ and ending with the block $\mathbf w_m$ is called a \emph{decomposition} of the word \textbf w. For a given word \textbf w, a letter $x\in\con(\mathbf w)$ and a natural number $i\le\occ_x(\mathbf w)$, we denote by $h_i(\mathbf w,x)$ the right-most divider of \textbf w that precedes the $i$th occurrence of $x$ in $\mathbf w$, and by $t(\mathbf w,x)$ the right-most divider of \textbf w that precedes the latest occurrence of $x$ in \textbf w.

\begin{example}
\label{example decompositions}
Let $\mathbf w=yxsxy^2tzy$. Then $\simple(\mathbf w)=\{s,t,z\}$ and $\mul(\mathbf w)=\{x,y\}$. Therefore, the decomposition of $\mathbf w$ has the form
\begin{equation}
\label{example 0-decomposition}
\lambda\cdot\underline{yx}\cdot s\cdot\underline{xy^2}\cdot t\cdot\underline\lambda\cdot z\cdot\underline{y}
\end{equation}
(here we underline blocks to distinguish them from dividers). Then we have that $h_1(\mathbf w,x)=h_1(\mathbf w,y)=h_1(\mathbf w,s)=t(\mathbf w,s)=\lambda$, $h_2(\mathbf w,x)=t(\mathbf w,x)=h_2(\mathbf w,y)=h_3(\mathbf w,y)=h_1(\mathbf w,t)=t(\mathbf w,t)=s$, $h_1(\mathbf w,z)=t(\mathbf w,z)=t$ and $h_4(\mathbf w,y)=t(\mathbf w, y)=z$.
\end{example}

Put
\begin{align*}
&\mathbf E=\var\{x^2\approx x^3,\,x^2y\approx xyx,\,x^2y^2\approx y^2x^2\},\\
&\mathbf F=\var\{\Phi,\, xyxz\approx xyxzx\}.
\end{align*}
The following statement implies, in particular, that $\mathbf E\subset\mathbf F$. 

\begin{lemma}[\mdseries{\!\!\cite[Propositions~4.2 and 6.9(i)]{Gusev-Vernikov-18}}]
\label{word problems}
A non-trivial identity $\mathbf u\approx \mathbf v$ holds:
\begin{itemize}
\item[\textup{(i)}]in the variety $\mathbf E$ if and only if
\begin{equation}
\label{sim(u)=sim(v) & mul(u)=mul(v)}
\simple(\mathbf u)=\simple(\mathbf v)\text{ and }\mul(\mathbf u)=\mul(\mathbf v)
\end{equation}
and
\begin{equation}
\label{h_1(u,x)=h_1(v,x)}
h_1(\mathbf u,x)= h_1(\mathbf v,x)\text{ for all }x\in \con(\mathbf u);
\end{equation}
\item[\textup{(ii)}]in the variety $\mathbf F$ if and only if 
\begin{equation}
\label{h_2(u,x)=h_2(v,x)}
h_2(\mathbf u,x)= h_2(\mathbf v,x)\text{ for all }x\in \con(\mathbf u)
\end{equation}
and the claims~\eqref{sim(u)=sim(v) & mul(u)=mul(v)} and~\eqref{h_1(u,x)=h_1(v,x)} are true.\qed
\end{itemize}
\end{lemma}

Lemma~\ref{word problems} implies the following two statements.

\begin{corollary}
\label{word problem F vee dual to E}
A non-trivial identity $\mathbf u\approx \mathbf v$ holds in the variety $\mathbf F\vee\overleftarrow{\mathbf E}$ if and only if
\begin{equation}
\label{t(u,x)=t(v,x)}
t(\mathbf u,x)= t(\mathbf v,x)\text{ for all }x\in \con(\mathbf u)
\end{equation}
and the claims~\eqref{sim(u)=sim(v) & mul(u)=mul(v)}--\eqref{h_2(u,x)=h_2(v,x)} are true.\qed
\end{corollary}

\begin{corollary}
\label{decompositions of u and v in E}
Let $\mathbf u\approx \mathbf v$ be an identity that holds in the variety $\mathbf E$. Suppose that 
\begin{equation}
\label{decomposition of u}
t_0\mathbf u_0 t_1 \mathbf u_1 \cdots t_m \mathbf u_m
\end{equation}
is the decomposition of $\mathbf u$. Then the decomposition of $\mathbf v$ has the form
\begin{equation}
\label{decomposition of v}
t_0\mathbf v_0 t_1 \mathbf v_1 \cdots t_m \mathbf v_m.
\end{equation}
\end{corollary}

\begin{proof}
In view of Lemma~\ref{word problems}(i), the claims~\eqref{sim(u)=sim(v) & mul(u)=mul(v)} and~\eqref{h_1(u,x)=h_1(v,x)} are true. Taking into account the claim~\eqref{sim(u)=sim(v) & mul(u)=mul(v)}, we obtain that $\simple(\mathbf v)=\{t_1,t_2,\dots,t_m\}$. Then the claim~\eqref{h_1(u,x)=h_1(v,x)} implies that $\mathbf v(t_1,t_2,\dots, t_m)=t_1t_2\cdots t_m$, and we are done.
\end{proof}

If $\mathbf u$ and $\mathbf v$ are words and $\varepsilon$ is an identity then we will write $\mathbf u\stackrel{\varepsilon}\approx\mathbf v$ in the case when the identity $\mathbf u\approx\mathbf v$ follows from $\varepsilon$.

\begin{lemma}
\label{basis for E vee dual E}
The identities
\begin{align}
\label{xyzx=xyxzx}
xyzx&\approx xyxzx,\\
\label{xxyy=yyxx}
x^2y^2&\approx y^2x^2
\end{align}
form an identity basis of the variety $\mathbf E\vee\overleftarrow{\mathbf E}$.
\end{lemma}

\begin{proof}
Consider the semigroup
$$
B_0 = \langle a, b, c \mid a^2 = a, b^2 = b, ab = ba = 0, ac = cb = c \rangle=\{a,b,c,0\}.
$$
It follows from~\cite[Proposition~1.7(i),(ii) and Figure~4]{Lee-08} that the variety $\mathbf E\vee\overleftarrow{\mathbf E}$ is generated by the monoid $B_0^1$, i.e., the semigroup $B_0$ with a new identity element adjoined. The identities~\eqref{xyzx=xyxzx} and
\begin{align}
\label{xzytxy=xzytyx}
xzytxy&\approx xzytyx,\\
\label{xyzxty=yxzxty}
xyzxty&\approx yxzxty,\\
\label{xzxyty=xzyxty}
xzxyty&\approx xzyxty
\end{align}
form an identity basis of the monoid $B_0^1$ by~\cite[Proposition~3.1(i)]{Edmunds-77}. We note that the identity~\eqref{xzytxy=xzytyx} follows from the identities~\eqref{xyzx=xyxzx} and~\eqref{xxyy=yyxx} because
$$
xzytxy\stackrel{\eqref{xyzx=xyxzx}}\approx xzytx^2y^2\stackrel{\eqref{xxyy=yyxx}}\approx  xzyty^2x^2\stackrel{\eqref{xyzx=xyxzx}}\approx xzytyx.
$$
Analogously, the identities~\eqref{xyzx=xyxzx} and~\eqref{xxyy=yyxx} imply the identities~\eqref{xyzxty=yxzxty} and~\eqref{xzxyty=xzyxty}. It remains to note that the identities~\eqref{xyzx=xyxzx} and~\eqref{xxyy=yyxx} hold in $\mathbf E\vee\overleftarrow{\mathbf E}$.
\end{proof}

The set of all letters that occur precisely $k$ times in a word \textbf w is denoted by $\con_k(\mathbf w)$.

\begin{lemma}
\label{basis for F vee dual E}
The identities~\eqref{xxyy=yyxx},~\eqref{xyzxty=yxzxty},~\eqref{xzxyty=xzyxty} and
\begin{align}
\label{xyx=xyxx}
xyx&\approx xyx^2,\\
\label{xyxztx=xyxzxtx}
xyxztx&\approx xyxzxtx
\end{align}
form an identity basis of the variety $\mathbf F\vee\overleftarrow{\mathbf E}$.
\end{lemma}

\begin{proof}
The identities~\eqref{xxyy=yyxx},~\eqref{xyzxty=yxzxty},~\eqref{xzxyty=xzyxty},~\eqref{xyx=xyxx}  and~\eqref{xyxztx=xyxzxtx} hold in $\mathbf F\vee\overleftarrow{\mathbf E}$ by Corollary~\ref{word problem F vee dual to E}. Note that~\eqref{xzytxy=xzytyx} follows from~\eqref{xxyy=yyxx} and~\eqref{xyx=xyxx} because
$$
xzytxy\stackrel{\eqref{xyx=xyxx}}\approx xzytx^2y^2\stackrel{\eqref{xxyy=yyxx}}\approx xzyty^2x^2\stackrel{\eqref{xyx=xyxx}}\approx xzytyx.
$$

Let $\mathbf u\approx \mathbf v$ be an identity that holds in $\mathbf F\vee\overleftarrow{\mathbf E}$. The identities~\eqref{xyx=xyxx} and~\eqref{xyxztx=xyxzxtx} allow us to assume that 
\begin{equation}
\label{mul=con3 for u and v}
\mul(\mathbf u)=\con_3(\mathbf u)\ \text{ and }\ \mul(\mathbf v)=\con_3(\mathbf v).
\end{equation}
In view of Corollary~\ref{decompositions of u and v in E}, if~\eqref{decomposition of u} is the decomposition of $\mathbf u$ then the decomposition of $\mathbf v$ has the form~\eqref{decomposition of v}. This fact, Corollary~\ref{word problem F vee dual to E} and the claim~\eqref{mul=con3 for u and v} imply that
\begin{equation}
\label{con_j(u_i)=con_j(v_i)}
\con_j(\mathbf u_i)=\con_j(\mathbf v_i) \ \text{ for any }\ i=0,1,\dots,m\ \text{ and }\ j=1,2,3.
\end{equation}
Note that the identity~\eqref{xzytxy=xzytyx} [respectively~\eqref{xyzxty=yxzxty}] allows us to swap the adjacent non-first [respectively the non-latest] occurrences of two multiple letters, while the identity~\eqref{xzxyty=xzyxty} allows us to swap a non-first occurrence and a non-latest occurrence of two multiple letters whenever these occurrences are adjacent to each other. So, since the claim~\eqref{con_j(u_i)=con_j(v_i)} is true, the identities~\eqref{xzytxy=xzytyx}--\eqref{xzxyty=xzyxty} imply the identities
$$
\mathbf u=t_0\mathbf u_0t_1\mathbf u_1\cdots t_m\mathbf u_m\approx t_0\mathbf v_0t_1\mathbf u_1\cdots t_m\mathbf u_m\approx\dots\approx t_0\mathbf v_0t_1\mathbf v_1\cdots t_m\mathbf v_m=\mathbf v,
$$
and we are done.
\end{proof}

\section{The subvariety lattice of $\mathbf J$}
\label{the subvariety lattice of J}

The trivial variety of monoids is denoted by $\mathbf T$, while $\mathbf{SL}$ denotes the variety of all semilattice monoids. Put also
\begin{align*}
&\mathbf C=\var\{x^2\approx x^3,\,xy\approx yx\},\\
&\mathbf D=\var\{x^2\approx x^3,\,x^2y\approx xyx\approx yx^2\},\\
&\mathbf H=\var\{\Phi,\,xyxztx\approx xyxzxtx,\, x^2yty\approx xyxty\approx yx^2ty\},\\
&\mathbf I=\var\{\Phi,\,xyxztx\approx xyxzxtx,\, xzxyty\approx xzyxty\}.
\end{align*}
The subvariety lattice of a monoid variety \textbf X is denoted by $L(\mathbf X)$.

The goal of this section is to prove the following

\begin{proposition}
\label{L(J)}
The lattice $L(\mathbf J)$ has the form shown in Fig.~\ref{pic L(J)}.
\end{proposition}

\begin{figure}[htb]
\unitlength=1mm
\linethickness{0.4pt}
\begin{center}
\begin{picture}(55,100)
\put(35,5){\circle*{1.33}}
\put(35,15){\circle*{1.33}}
\put(35,25){\circle*{1.33}}
\put(35,35){\circle*{1.33}}
\put(45,45){\circle*{1.33}}
\put(25,45){\circle*{1.33}}
\put(35,55){\circle*{1.33}}
\put(15,55){\circle*{1.33}}
\put(25,65){\circle*{1.33}}
\put(25,75){\circle*{1.33}}
\put(25,85){\circle*{1.33}}
\put(25,95){\circle*{1.33}}

\put(35,5){\line(0,1){30}}
\put(35,35){\line(-1,1){20}}
\put(35,35){\line(1,1){10}}
\put(25,45){\line(1,1){10}}
\put(25,45){\line(1,1){10}}
\put(15,55){\line(1,1){10}}
\put(45,45){\line(-1,1){20}}
\put(25,65){\line(0,1){30}}

\put(35,2){\makebox(0,0)[cc]{\textbf T}}
\put(37,15){\makebox(0,0)[lc]{\textbf{SL}}}
\put(37,25){\makebox(0,0)[lc]{$\mathbf C$}}
\put(37,35){\makebox(0,0)[lc]{$\mathbf D$}}
\put(23,45){\makebox(0,0)[rc]{$\mathbf E$}}
\put(13,55){\makebox(0,0)[rc]{$\mathbf F$}}
\put(47,45){\makebox(0,0)[lc]{$\overleftarrow{\mathbf E}$}}
\put(37,55){\makebox(0,0)[lc]{$\mathbf E\vee\overleftarrow{\mathbf E}$}}
\put(27,65){\makebox(0,0)[lc]{$\mathbf F\vee\overleftarrow{\mathbf E}$}}
\put(27,75){\makebox(0,0)[lc]{$\mathbf H$}}
\put(27,85){\makebox(0,0)[lc]{$\mathbf I$}}
\put(27,95){\makebox(0,0)[lc]{$\mathbf J$}}
\end{picture}
\end{center}
\caption{The lattice $L(\mathbf J)$}
\label{pic L(J)}
\end{figure}
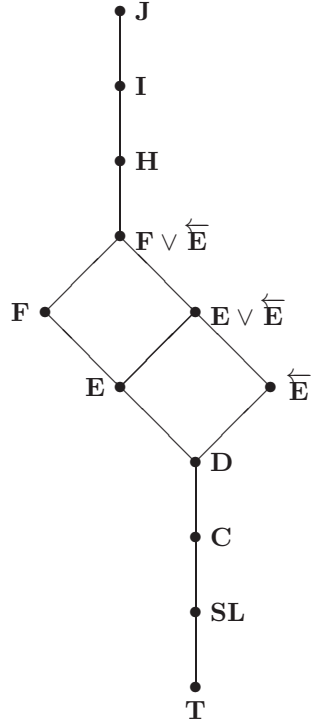

To verify Proposition~\ref{L(J)}, we need several assertions.

\begin{lemma}
\label{L(F vee dual E) and [F vee dual E,X]}
Let $\mathbf X$ be a monoid variety that satisfies the identities~\eqref{xxyy=yyxx}, \eqref{xyx=xyxx}, \eqref{xyxztx=xyxzxtx} and 
\begin{equation}
\label{xyzxy=yxzxy}
xyzxy\approx yxzxy.
\end{equation}
Then the lattice $L(\mathbf X)$ is the set-theoretical union of the lattice $L(\mathbf F\vee\overleftarrow{\mathbf E})$ and the interval $[\mathbf F\vee\overleftarrow{\mathbf E},\mathbf X]$. The lattice $L(\mathbf F\vee\overleftarrow{\mathbf E})$ has the form shown in Fig.~\ref{pic L(J)}.
\end{lemma}

\begin{proof}
Let $\mathbf V$ be a subvariety of the variety $\mathbf X$. We need to verify that if $\mathbf F\vee\overleftarrow{\mathbf E}\nsubseteq \mathbf V$ then $\mathbf V$ coincides with one of the varieties $\mathbf T$, $\mathbf{SL}$, $\mathbf C$, $\mathbf D$, $\mathbf E$, $\overleftarrow{\mathbf E}$, $\mathbf E\vee \overleftarrow{\mathbf E}$ and $\mathbf F$. Clearly, $\mathbf V$ does not contain either $\mathbf F$ or $\overleftarrow{\mathbf E}$. A variety of monoids is called \emph{completely regular} if it consists of \emph{completely regular monoids}~(i.e., unions of groups). If $\mathbf V$ is completely regular then it is a variety of \emph{bands}, i.e. idempotent monoids because every aperiodic completely regular variety is a variety of bands. Evidently, every variety of bands with the identity~\eqref{xxyy=yyxx} is commutative. Therefore, $\mathbf V$ is one of the varieties $\mathbf T$ or $\mathbf{SL}$. So, we can assume that $\mathbf V$ is non-completely regular. 

Suppose that $\overleftarrow{\mathbf E} \nsubseteq \mathbf V$. It is verified in~\cite[the dual to Lemma~4.3]{Gusev-Vernikov-18} that if $\mathbf Y$ is a non-completely regular variety of monoids that satisfies the identity
\begin{equation}
\label{xx=xxx}
x^2\approx x^3
\end{equation}
and does not contain the variety $\overleftarrow{\mathbf E}$ then $\mathbf Y$ satisfies the identity
\begin{equation}
\label{xxy=xxyxx}
x^2y\approx x^2yx^2.
\end{equation}
This fact implies that the identity~\eqref{xxy=xxyxx} holds in $\mathbf V$. Then the identities
$$
xyxz\stackrel{\eqref{xyx=xyxx}}\approx xyx^2z\stackrel{\eqref{xxy=xxyxx}}\approx xyx^2zx^2\stackrel{\eqref{xyx=xyxx}}\approx xyxzx
$$  
hold in $\mathbf V$. Thus $\mathbf V\subseteq\mathbf F$. It is proved in~\cite[Proposition~6.1]{Gusev-Vernikov-18} that the lattice $L(\mathbf F)$ has the form shown in Fig.~\ref{pic L(J)}. Hence $\mathbf V$ is one of the varieties $\mathbf C$, $\mathbf D$, $\mathbf E$ or $\mathbf F$. So, we can assume that $\overleftarrow{\mathbf E} \subseteq \mathbf V$. 

It follows that $\mathbf F \nsubseteq \mathbf V$. Then there exists an identity $\mathbf u\approx \mathbf v$ that holds in $\mathbf V$ but does not hold in $\mathbf F$. If~\eqref{decomposition of u} is the decomposition of $\mathbf u$ then the decomposition of $\mathbf v$ has the form~\eqref{decomposition of v} by the dual to Corollary~\ref{decompositions of u and v in E}. Lemma~\ref{word problems}(ii) and the dual to Lemma~\ref{word problems}(i) imply that one of the claims~\eqref{h_1(u,x)=h_1(v,x)} and~\eqref{h_2(u,x)=h_2(v,x)} is false. Then there are a letter $x\in\mul(\mathbf u)$ and $k\in\{1,2\}$ such that $h_k(\mathbf u,x)\ne h_k(\mathbf v, x)$. If $k=1$ then we multiply the identity $\mathbf u\approx \mathbf v$ by $xt$ on the left where $t\notin\con(\mathbf u)$. So, we can believe that $k=2$. Suppose that $h_2(\mathbf u,x)=t_i$ and $h_2(\mathbf v,x)=t_j$ where $i\ne j$. We can assume without any loss that $i>j$. Then $\mathbf V$ satisfies the identity
$$
\mathbf u(x,t_i)=xt_ix^p\approx x^qt_ix^r=\mathbf v(x,t_i)
$$
where $p\ge1$, $q\ge 2$ and $r\ge 0$. If $r<2$ or $p<2$ then we multiply this identity by $x^2$ on the right. Taking into account the identity~\eqref{xx=xxx}, we get that $\mathbf V$ satisfies the identity
\begin{equation}
\label{xyxx=xxyxx}
xyx^2\approx x^2yx^2.
\end{equation}
Then the identities 
$$
xyzx\stackrel{\eqref{xyx=xyxx}}\approx xyzx^2\stackrel{\eqref{xyxx=xxyxx}}\approx x^2yzx^2\stackrel{\eqref{xyxztx=xyxzxtx}}\approx x^2yxzx^2\stackrel{\eqref{xyxx=xxyxx}}\approx xyxzx^2\stackrel{\eqref{xyx=xyxx}}\approx xyxzx
$$
hold in $\mathbf V$. In view of Lemma~\ref{basis for E vee dual E}, we have that $\mathbf V\subseteq\mathbf E\vee\overleftarrow{\mathbf E}$. It is proved in~\cite[Section~5]{Lee-08} that the lattice $L(\mathbf E\vee\overleftarrow{\mathbf E})$ has the form shown in Fig.~\ref{pic L(J)}. This fact implies that $\mathbf V$ is one of the varieties $\overleftarrow{\mathbf E}$ and $\mathbf E\vee\overleftarrow{\mathbf E}$. 
\end{proof}

\begin{lemma}
\label{w=v_1aav_2v_3}
Let $\mathbf w=\mathbf v_1a\mathbf v_2a\mathbf v_3$ where $\mathbf v_1$, $\mathbf v_2$ and $\mathbf v_3$ are possibly empty words. Suppose that $\con(\mathbf v_2)\subseteq \mul(\mathbf w)$. Then $\mathbf J$ satisfies the identity $\mathbf w\approx\mathbf v_1a^2\mathbf v_2\mathbf v_3$.
\end{lemma}

\begin{proof}
Put $X=\con(\mathbf v_2) \setminus\con(\mathbf v_3)$. We use induction on the cardinality of the set $X$ and aim to verify that $\mathbf J$ satisfies the identity $\mathbf w\approx\mathbf v_1a^2\mathbf v_2\mathbf v_3$.

\smallskip
 
\emph{Induction base}. Let the set $X$ be empty. Then $\con(\mathbf v_2)\subseteq\con(\mathbf v_3)$. We can believe that $\occ_x(\mathbf v_2)\le\occ_x(\mathbf v_3)$ for every $x\in\con(\mathbf v_2)$ because $\mathbf J$ satisfies the identity~\eqref{xyx=xyxx}. Then we can rename the letters and assume that $\mathbf v_2=z_{1\pi}z_{2\pi}\cdots z_{n\pi}$ for some $n$ and $\pi\in S_n$ and the latest occurrence of $z_i$ precedes the latest occurrence of $z_j$ in $\mathbf v_3$ whenever $i<j$. It follows that there are letters $t_1,t_2,\dots,t_n\notin \con(\mathbf w)$, a mapping $\xi$ from $\{t_1,t_2,\dots,t_n\}$ to $F^1$ and a word $\mathbf v\in F^1$ such that
\begin{align*}
\mathbf v_1a\mathbf v_2\,a\,\mathbf v_3={}&\mathbf v_1a\,z_{1\pi}z_{2\pi}\cdots z_{n\pi}\, a\,\biggl(\,\prod_{i=1}^n \xi(t_i)z_i\biggr)\mathbf v,\\
\mathbf v_1a^2\mathbf v_2\,\mathbf v_3={}&\mathbf v_1a^2\,z_{1\pi}z_{2\pi}\cdots z_{n\pi}\,\biggl(\,\prod_{i=1}^n \xi(t_i)z_i\biggr)\mathbf v.
\end{align*}
We can extend $\xi$ to an endomorphis of $F^1$ so that $\xi(z_i)=z_i$ for all $i=1,2,\dots,n$. So, we can assume without any loss that $\xi\in\End(F^1)$ and $\xi(z_i)=z_i$ for all $i=1,2,\dots,n$. Then
$$
\mathbf w =\mathbf v_1a\mathbf v_2\,a\,\mathbf v_3=\mathbf v_1\,\xi(\mathbf w_n[\pi])\,\mathbf v\ \text{ and }
\mathbf v_1a^2\mathbf v_2\,\mathbf v_3=\mathbf v_1\,\xi(\mathbf w_n'[\pi])\,\mathbf v,
$$
whence $\mathbf J$ satisfies the identity $\mathbf w=\mathbf v_1a^2\mathbf v_2\mathbf v_3$.

\smallskip
 
\emph{Induction step}. Let now $X$ be non-empty. Then there is a letter $x\in\con(\mathbf v_2)$ and the possibly empty words $\mathbf v_2'$ and $\mathbf v_2''$ such that $\mathbf v_2=\mathbf v_2'x\mathbf v_2''$, $\con(\mathbf v_2'')\subseteq \con(\mathbf v_3)$ and $x\notin\con(\mathbf v_3)$. Taking into account that $x\in\mul(\mathbf w)$, we get that $x\in\con(\mathbf v_1a\mathbf v_2')$. Then $\mathbf J$ satisfies the identities
\begin{align*}
\mathbf w{}&=\mathbf v_1a\mathbf v_2'x\mathbf v_2''a\mathbf v_3\\
&\approx\mathbf v_1a\mathbf v_2'x^2\mathbf v_2''a\mathbf v_3&&\textup{by the identity~\eqref{xyx=xyxx}}\\
&\approx \mathbf v_1a\mathbf v_2'x\mathbf v_2''xa\mathbf v_3&&\textup{by the induction assumption}\\
&\approx \mathbf v_1a\mathbf v_2'x\mathbf v_2''x^2a^2\mathbf v_3&&\textup{by the identity~\eqref{xyx=xyxx}}\\
&\approx \mathbf v_1a\mathbf v_2'x\mathbf v_2''a^2x^2\mathbf v_3&&\textup{by the identity~\eqref{xxyy=yyxx}}\\
&\approx \mathbf v_1a\mathbf v_2'x\mathbf v_2''ax\mathbf v_3&&\textup{by the identity~\eqref{xyx=xyxx}}\\
&\approx \mathbf v_1a^2\mathbf v_2'x\mathbf v_2''x\mathbf v_3&&\textup{by the induction assumption}\\
&\approx \mathbf v_1a^2\mathbf v_2'x^2\mathbf v_2''\mathbf v_3&&\textup{by the induction assumption}
\end{align*}
\begin{align*}
&\approx \mathbf v_1a^2\mathbf v_2'x\mathbf v_2''\mathbf v_3&&\textup{by the identity~\eqref{xyx=xyxx}}\\
&=\mathbf v_1a^2\mathbf v_2\mathbf v_3,
\end{align*}
and we are done.
\end{proof}

Let $\mathbf w$ be a word and $x,y\in \mul(\mathbf w)$. Suppose that $t_0\mathbf w_0 t_1 \mathbf w_1 \cdots t_m \mathbf w_m$ is the decomposition of $\mathbf w$ and $t_i=h_2(\mathbf w,x)$, $t_{i'}=h_2(\mathbf w,y)$, $t_j=t(\mathbf w,x)$, $t_{j'}=t(\mathbf w,y)$ for some $0\le i\le j\le m$ and $0\le i'\le j'\le m$. The letters $x$ and $y$ are said to be \emph{integrated} in the word $\mathbf w$ if either $i'\le i\le j'$ or $i\le i'\le j$.

\begin{lemma}
\label{w=w'abw''}
Let $\mathbf w=\mathbf w'ab\mathbf w''$ where $\mathbf w'$ and $\mathbf w''$ are possibly empty words, $a$ and $b$ are multiple letters in $\mathbf w$. Suppose that one of the following holds:
\begin{itemize}
\item[\textup{(i)}] the letters $a$ and $b$ are integrated in $\mathbf w$;
\item[\textup{(ii)}] $\mathbf w'=\mathbf v_1b\mathbf v_2$ for some words $\mathbf v_1$ and $\mathbf v_2$ such that $\con(\mathbf v_2)\subseteq \mul(\mathbf w)$.
\end{itemize}
Then $\mathbf J$ satisfies the identity $\mathbf w\approx\mathbf w'ba\mathbf w''$.
\end{lemma}

\begin{proof}
(i) If $a,b\in\con(\mathbf w')$ then
\begin{equation}
\label{w'bbaaw''=w'aabbw''}
\mathbf w= \mathbf w'ab\mathbf w''\stackrel{\eqref{xyx=xyxx}}\approx\mathbf w'a^2b^2\mathbf w''\stackrel{\eqref{xxyy=yyxx}}\approx\mathbf w'b^2a^2\mathbf w''\stackrel{\eqref{xyx=xyxx}}\approx\mathbf w'ba\mathbf w'',
\end{equation}
and we are done. Thus, we can assume without loss of generality that $a\notin \con(\mathbf w')$. Then $a\in \con(\mathbf w'')$. If $b\notin\con(\mathbf w'')$ then $\mathbf J$ satisfies the identities 
$$
\mathbf w=\mathbf w'ab\mathbf w''\stackrel{\eqref{xyx=xyxx}}\approx\mathbf w'ab^2\mathbf w''.
$$
This fact allows us to assume that $b\in\con(\mathbf w'')$.

If there is some occurrence of $b$ between the second and the latest occurrences of $a$ in $\mathbf w$ then $\mathbf w''=\mathbf v_1a\mathbf v_2b\mathbf v_3a\mathbf v_4$ for some possibly empty words $\mathbf v_1,\mathbf v_2,\mathbf v_3,\mathbf v_4$. Then the identities
\begin{align*}
\mathbf w&\stackrel{\phantom{\eqref{xyzxy=yxzxy}}}=\mathbf w'ab\mathbf w'=\mathbf w'ab\mathbf v_1a\mathbf v_2b\mathbf v_3a\mathbf v_4\stackrel{\eqref{xyxztx=xyxzxtx}}\approx
\mathbf w'ab\mathbf v_1a\mathbf v_2ab\mathbf v_3a\mathbf v_4\\
&\stackrel{\eqref{xyzxy=yxzxy}}\approx\mathbf w'ba\mathbf v_1a\mathbf v_2ab\mathbf v_3a\mathbf v_4\stackrel{\eqref{xyxztx=xyxzxtx}}\approx
\mathbf w'ba\mathbf v_1a\mathbf v_2b\mathbf v_3a\mathbf v_4=\mathbf w'ba\mathbf w''
\end{align*}
hold in $\mathbf J$. So, we can assume that there are no occurrences of $b$ between the second and the latest occurrences of $a$ in $\mathbf w$. Analogously, the second and the latest occurrences of $a$ in $\mathbf w$ do not lie between the second and the latest occurrences of $b$ in $\mathbf w$. Then either the latest occurrence of $a$ precedes the second occurrence of $b$ in $\mathbf w$ or the latest occurrence of $b$ precedes the second occurrence of $a$. Since the letters $a$ and $b$ are integrated in the word $\mathbf w$, either the latest occurrence of $a$ and the second occurrence of $b$ in $\mathbf w$ or the latest occurrence of $b$ and the second occurrence of $a$ in $\mathbf w$ lie in the same block. 

If the latest occurrence of $a$ and the second occurrence of $b$ in $\mathbf w$ lie in the same block then $\mathbf w''=\mathbf v_1a\mathbf v_2b\mathbf v_3$ for some possibly empty words $\mathbf v_1,\mathbf v_2,\mathbf v_3$ such that $\con(\mathbf v_2)\subseteq\mul(\mathbf w)$ and $a,b\notin\con(\mathbf v_2)$. Then the identities 
\begin{align*}
\mathbf w&{}=\mathbf w'ab\mathbf v_1a\mathbf v_2b\mathbf v_3\\
&{}\approx\mathbf w'ab\mathbf v_1a^2\mathbf v_2b\mathbf v_3&&\textup{by the identity~\eqref{xyx=xyxx}}\\
&\approx \mathbf w'ab\mathbf v_1a\mathbf v_2ab\mathbf v_3&&\textup{by Lemma~\ref{w=v_1aav_2v_3}}\\
&\approx \mathbf w'ba\mathbf v_1a\mathbf v_2ab\mathbf v_3&&\textup{by the identity~\eqref{xyzxy=yxzxy}}\\
&\approx \mathbf w'ba\mathbf v_1a^2\mathbf v_2b\mathbf v_3&&\textup{by Lemma~\ref{w=v_1aav_2v_3}}
\end{align*}
\begin{align*}
&\approx \mathbf w'ba\mathbf v_1a\mathbf v_2b\mathbf v_3&&\textup{by the identity~\eqref{xyx=xyxx}}\\
&=\mathbf w'ba\mathbf w''
\end{align*}
hold in $\mathbf J$. It follows that $\mathbf J$ satisfies the identity $\mathbf w\approx\mathbf w'ba\mathbf w''$.

The case when the latest occurrence of $b$ and the second occurrence of $a$ in $\mathbf w$ lie in the same block is considered similarly.

\smallskip

(ii) Since $a\in\mul(\mathbf w)$, Lemma~\ref{w=v_1aav_2v_3} implies that $\mathbf J$ satisfies the identities
$$
\mathbf w=\mathbf v_1b\mathbf v_2 ab\mathbf w''\approx \mathbf v_1b^2\mathbf v_2 a\mathbf w''\approx \mathbf v_1b\mathbf v_2 ba\mathbf w''=\mathbf w'ba\mathbf w'',
$$
i.e., the identity $\mathbf w\approx \mathbf w'ba\mathbf w''$.
\end{proof}

The following assertion is evident.

\begin{remark}
\label{subwords of v}
Let 
$
\mathbf v\in \{xzyx^pty^q,\, yx^pty^q,\,xyzx^pty^q\mid p,q\in\mathbb N\}.
$
If $a,b\in\con(\mathbf v)$ and $a\ne b$ then the subword $ab$ of the word $\mathbf v$ has exactly one occurrence in this word.\qed
\end{remark}

\begin{lemma}
\label{w = ... in H,I,J}
Let $p_1$ and $p_2$ be natural numbers.
\begin{itemize}
\item[\textup{(i)}] If the variety $\mathbf J$ satisfies an identity $xzyx^{p_1}ty^{p_2}\approx \mathbf w$ then $\mathbf w = xzyx^{q_1}ty^{q_2}$ for some natural numbers $q_1$ and $q_2$.
\item[\textup{(ii)}] If the variety $\mathbf I$ satisfies an identity $yx^{p_1}ty^{p_2}\approx \mathbf w$ then $\mathbf w = yx^{q_1}ty^{q_2}$ for some natural numbers $q_1$ and $q_2$.
\item[\textup{(iii)}] If the variety $\mathbf H$ satisfies an identity $xyzx^{p_1}ty^{p_2}\approx \mathbf w$ then $\mathbf w = xyzx^{q_1}ty^{q_2}$ for some natural numbers $q_1$ and $q_2$.
\end{itemize}
\end{lemma}

\begin{proof}
Put 
$$
\Psi=\{\Phi,\,xyxztx\approx xyxzxtx,\,\mathbf w_n[\pi]\approx \mathbf w_n'[\pi]\mid n\in \mathbb N,\ \pi\in S_n\}.
$$
 
(i) Put $\mathbf v = xzyx^{p_1}ty^{p_2}$. By Lemma \ref{deduction} and induction, we can reduce our considerations to the case when either $\mathbf v= \mathbf a\xi(\mathbf s)\mathbf b$, $\mathbf w= \mathbf a\xi(\mathbf t)\mathbf b$ or $\mathbf v= \mathbf a\xi(\mathbf t)\mathbf b$, $\mathbf w= \mathbf a\xi(\mathbf s)\mathbf b$ for some $\mathbf a,\mathbf b\in F^1$, $\xi\in\End(F^1)$ and $\mathbf s\approx \mathbf t\in\Psi$. We can assume without loss of generality that the words $\mathbf v$ and $\mathbf w$ are different.

If $\xi(x)$ is the empty word then $\xi(\mathbf s)=\xi(\mathbf t)$, but this is impossible because $\mathbf v\ne\mathbf w$. Thus, $\xi(x)\ne\lambda$. Then, since $\con(\xi(x))\subseteq\mul(\xi(\mathbf s))\subseteq\mul(\mathbf v)$, Remark~\ref{subwords of v} implies that $\xi(x)=c^k$ for some $k\in\mathbb N$ and $c\in\{x,y\}$.

The identity~\eqref{xyxztx=xyxzxtx} allows us to add and delete the occurrences of the letter $x$ between the second and the latest occurrences of this letter, while the identity~\eqref{xyx=xyxx} allows us to add and delete the occurrences of the letter $x$ next to the non-first occurrence of this letter. This implies that if $\mathbf s\approx \mathbf t$ coincides with one of the identities~\eqref{xyx=xyxx} or~\eqref{xyxztx=xyxzxtx} then $\mathbf w = xzyx^{q_1}ty^{q_2}$ for some $q_1,q_2\in\mathbb N$.

Suppose now that $\mathbf s\approx \mathbf t\in\{x^2y^2\approx y^2x^2,\,xyzxy\approx yxzxy\}$. Note that $\xi(y)\ne\lambda$ because the identity $\mathbf v\approx\mathbf w$ is non-trivial. Then $\xi(y)=d^{k'}$ for some letter $d$ and some natural number $k'$ by Remark~\ref{subwords of v}. Since the words $\mathbf v$ and $\mathbf w$ are different, $c\ne d$. Corollary~\ref{word problem F vee dual to E} implies that $\mathbf w =xzx^ryx^{q_1}ty^{q_2}$ for some numbers $q_1,q_2$ and $r$. The case when the identity $\mathbf s\approx \mathbf t$ coincides with the identity~\eqref{xxyy=yyxx} is impossible because the word $\mathbf v$ does not contain any subword of the form $c^{2k}d^{2k'}$ and $d^{2k'}c^{2k}$. The identity $\mathbf s\approx \mathbf t$ can not also coincide with the identity $xyzxy\approx yxzxy$ because the words $\mathbf v$ and $\mathbf w$ may contain at most one occurrence of the word $c^kd^{k'}$.

So, it remains to consider the case when the identity $\mathbf s\approx \mathbf t$ coincides with the identity $\mathbf w_n[\pi]\approx \mathbf w_n'[\pi]$ for some $n\in \mathbb N$ and $\pi\in S_n$. If $\mathbf v= \mathbf a\xi(\mathbf s)\mathbf b$ and $\mathbf w= \mathbf a\xi(\mathbf t)\mathbf b$ then
$$
xzyx^{p_1}ty^{p_2}=\mathbf a\,c^k\xi(z_{1\pi}z_{2\pi}\cdots z_{n\pi})c^k\xi\biggl(\,\prod_{i=1}^n t_iz_i^{\ell_i}\biggr)\,\mathbf b,
$$
whence $\xi(z_{1\pi}z_{2\pi}\cdots z_{n\pi})=c^h$ for some $h\ge0$. This contradicts the fact that the identity $\mathbf v\approx\mathbf w$ is non-trivial. If $\mathbf v= \mathbf a\xi(\mathbf t)\mathbf b$ and $\mathbf w= \mathbf a\xi(\mathbf s)\mathbf b$ then
$$
xzyx^{p_1}ty^{p_2}=\mathbf a\,c^{2k}\xi(z_{1\pi}z_{2\pi}\cdots z_{n\pi})\xi\biggl(\,\prod_{i=1}^n t_iz_i^{\ell_i}\biggr)\,\mathbf b,
$$
whence $\xi(z_{1\pi}z_{2\pi}\cdots z_{n\pi})=c^h$ for some $h\ge0$. We obtain a contradiction with the inequality $\mathbf v\ne\mathbf w$ again.

\smallskip

(ii) Put $\mathbf v = yx^{p_1}ty^{p_2}$. By Lemma \ref{deduction} and induction, we can reduce our considerations to the case when either $\mathbf v= \mathbf a\xi(\mathbf s)\mathbf b$, $\mathbf w= \mathbf a\xi(\mathbf t)\mathbf b$ or $\mathbf v= \mathbf a\xi(\mathbf t)\mathbf b$, $\mathbf w= \mathbf a\xi(\mathbf s)\mathbf b$ for some $\mathbf a,\mathbf b\in F^1$, $\xi\in\End(F^1)$ and $\mathbf s\approx \mathbf t\in\{\Psi,\, xzxyty\approx xzyxty\}$. We can assume without loss of generality that the words $\mathbf v$ and $\mathbf w$ are different.

Suppose that the $\mathbf s\approx \mathbf t\in \Psi$. Corollary~\ref{word problem F vee dual to E} implies that $\mathbf w = x^ryx^{q_1}ty^{q_2}$ for some numbers $q_1,q_2$ and $r$. If $r>0$ then we multiply the identity $\mathbf v\approx \mathbf w$ by $xz$ on the left and obtain a contradiction with the claim~(i).

So, we can assume that $\mathbf s\approx \mathbf t$ equals the identity~\eqref{xzxyty=xzyxty}. Since the identity $\mathbf v\approx \mathbf w$ is non-trivial, we have that $\xi(x)\ne\lambda$, $\xi(y)\ne\lambda$ and $\xi(x)\ne\xi(y)$. Then $\xi(x)=c^k$ and $\xi(y)=d^k$ for some $k,k'\in\mathbb N$ by Remark~\ref{subwords of v}. If $\mathbf v= \mathbf a\xi(\mathbf s)\mathbf b$ and $\mathbf w= \mathbf a\xi(\mathbf t)\mathbf b$ then
$$
\mathbf v=yx^{p_1}ty^{p_2}=\mathbf a c^k\xi(z)c^kd^{k'}\xi(t)d^{k'}\mathbf b.
$$
Then $\{x,y\}=\{c,d\}$. The case when $y=c$ and $x=d$ is impossible because the second occurrence of $y$ is preceded the latest occurrence of $x$ in $\mathbf v$. So, $x=c$ and $y=d$. This implies that $\xi(z)=\lambda$ and, therefore, $yx^{p_1}ty^{p_2}=\mathbf ax^{2k}y^{k'}\xi(t)y^{k'}\mathbf b$, a contradiction. If $\mathbf v= \mathbf a\xi(\mathbf t)\mathbf b$ and $\mathbf w= \mathbf a\xi(\mathbf s)\mathbf b$ then
$$
\mathbf v=yx^{p_1}ty^{p_2}=\mathbf a c^k\xi(z)d^{k'}c^k\xi(t)d^{k'}\mathbf b.
$$
But this equality is impossible too. The claim~(ii) is proved.

\smallskip

(iii) Put $\mathbf v = xyzx^{p_1}ty^{p_2}$. As in the proof of the claims~(i) and~(ii), by Lemma \ref{deduction} and induction, we can reduce our considerations to the case when either $\mathbf v= \mathbf a\xi(\mathbf s)\mathbf b$, $\mathbf w= \mathbf a\xi(\mathbf t)\mathbf b$ or $\mathbf v= \mathbf a\xi(\mathbf t)\mathbf b$, $\mathbf w= \mathbf a\xi(\mathbf s)\mathbf b$ for some $\mathbf a,\mathbf b\in F^1$, $\xi\in\End(F^1)$ and $\mathbf s\approx \mathbf t\in\{\Psi,\, xyxty\approx yx^2ty\}$. We can assume without loss of generality that the words $\mathbf v$ and $\mathbf w$ are different.

Suppose that $\mathbf s\approx \mathbf t\in \Psi$. Corollary~\ref{word problem F vee dual to E} implies that 
$$
\mathbf w \in\{xyzx^{q_1}ty^{q_2},\, yxzx^{q_1}ty^{q_2}\mid q_1,q_2\in\mathbb N\}.
$$
If $\mathbf w= yxzx^{q_1}ty^{q_2}$ for some $q_1,q_2\in\mathbb N$ then we substitute~1 for $z$ in the identity $\mathbf v\approx \mathbf w$ and obtain a contradiction with the claim~(ii).

So, we can assume that $\mathbf s\approx \mathbf t$ coincides with the identity~\eqref{xyxty=yxxty}. Since the identity $\mathbf v\approx \mathbf w$ is non-trivial, we have that $\xi(x)\ne\lambda$, $\xi(y)\ne\lambda$ and $\xi(x)\ne\xi(y)$. Then $\xi(x)=c^k$ and $\xi(y)=d^k$ for some $k,k'\in\mathbb N$ by Remark~\ref{subwords of v}. If $\mathbf v= \mathbf a\xi(\mathbf s)\mathbf b$ and $\mathbf w= \mathbf a\xi(\mathbf t)\mathbf b$ then
$$
\mathbf v= xyzx^{p_1}ty^{p_2}=\mathbf a c^kd^{k'}c^k\xi(t)d^{k'}\mathbf b.
$$
It is easy to see that this equality is impossible. If $\mathbf v= \mathbf a\xi(\mathbf t)\mathbf b$ and $\mathbf w= \mathbf a\xi(\mathbf s)\mathbf b$ then
$$
\mathbf v= xyzx^{p_1}ty^{p_2}=\mathbf a d^{k'}c^{2k}\xi(t)d^{k'}\mathbf b.
$$
But this equality is impossible too. The claim~(iii) is proved.
\end{proof}

If $x\in \con(\mathbf w)$ and $i\le \occ_x(\mathbf w)$ then $\ell_i(\mathbf w,x)$ denotes the length of the minimal prefix $\mathbf p$ of $\mathbf w$ with $\occ_x(\mathbf p)=i$.

\begin{proof}[Proof of Proposition~\ref{L(J)}.]
In view of Lemma~\ref{L(F vee dual E) and [F vee dual E,X]}, it remains to verify that the interval $[\mathbf F\vee\overleftarrow{\mathbf E},\mathbf J]$ is the chain $\mathbf F\vee\overleftarrow{\mathbf E}\subset\mathbf H\subset\mathbf I\subset\mathbf J$. Lemmas~\ref{basis for F vee dual E} and~\ref{w = ... in H,I,J} imply that $\mathbf F\vee\overleftarrow{\mathbf E}\subset\mathbf H\subset\mathbf I\subset\mathbf J$. So, it remains to verify that if $\mathbf V$ is a monoid variety such that $\mathbf F\vee\overleftarrow{\mathbf E}\subseteq\mathbf V\subset\mathbf J$ then $\mathbf V$ coincides with one of the varieties $\mathbf F\vee\overleftarrow{\mathbf E}$, $\mathbf H$ or $\mathbf I$. Since $\mathbf V\subset \mathbf J$, there is a non-trivial identity $\mathbf u\approx \mathbf v$ that holds in $\mathbf V$ but does not hold in $\mathbf J$. The identities~\eqref{xyx=xyxx} and~\eqref{xyxztx=xyxzxtx} allow us to assume that the claim~\eqref{mul=con3 for u and v} is true. In view of Corollary~\ref{decompositions of u and v in E} and inclusion $\mathbf E\subseteq\mathbf V$, if~\eqref{decomposition of u} is the decomposition of $\mathbf u$ then the decomposition of $\mathbf v$ has the form~\eqref{decomposition of v}. This fact, Corollary~\ref{word problem F vee dual to E} and the claim~\eqref{mul=con3 for u and v} imply that the claim~\eqref{con_j(u_i)=con_j(v_i)} is true. Since the identity $\mathbf u\approx \mathbf v$ is non-trivial, there is $0\le i\le m$ such that $\mathbf u_i\ne \mathbf v_i$. Let $\mathbf p$ be the greatest common prefix of the words $\mathbf u_i$ and $\mathbf v_i$. Suppose that $\mathbf u_i=\mathbf px\mathbf u_i'$ for some letter $x$ and some word $\mathbf u_i'$. The claim~\eqref{con_j(u_i)=con_j(v_i)} implies that there are words $\mathbf a,\mathbf b$ and the letter $y$ such that $\mathbf v_i=\mathbf p\mathbf ayx\mathbf b$ and $x\notin\con(\mathbf a y)$. We note also that $y\in\con(\mathbf u_i')$ by the claim~\eqref{con_j(u_i)=con_j(v_i)}.

By induction we can assume without loss of generality that $\mathbf J$ violates the identity
\begin{equation}
\label{..ayxb..=..axyb..}
\mathbf v\approx \mathbf v'\,\mathbf p\mathbf a\,xy\,\mathbf b\,\mathbf v'',
\end{equation}
where $\mathbf v'= \mathbf v_0t_1\mathbf v_1t_2\cdots \mathbf v_{i-1}t_i$ and $\mathbf v''= t_{i+1}\mathbf v_{i+1}\cdots t_m\mathbf v_m$. Then the letters $x$ and $y$ are non-integrated in the word $\mathbf v$ by Lemma~\ref{w=w'abw''}(i). Then either the third occurrence of $x$ precedes the second occurrence of $y$ in $\mathbf v$ or the third occurrence of $y$ precedes the second occurrence of $x$.

\medskip

\emph{Case} 1: the third occurrence of $x$ precedes the second occurrence of $y$ in $\mathbf v$. Then there is $j$ such that $\ell_3(\mathbf v, x)<\ell_1(\mathbf v, t_j)<\ell_2(\mathbf v, y)$. In view of Corollary~\ref{word problem F vee dual to E}, $\ell_3(\mathbf u, x)<\ell_1(\mathbf u, t_j)<\ell_2(\mathbf u, y)$. It follows that $y\notin\con(\mathbf v'\mathbf p\mathbf a)$.

\smallskip

First, we are going to verify that $\mathbf V\subseteq\mathbf I$. Suppose that $x\notin\con(\mathbf v'\mathbf p)$. Then $\mathbf V$ satisfies the identities 
$$
x^syx^rt_jy^2=\mathbf u(x,y,t_j)\approx \mathbf v(x,y,t_j)=yx^3ty^2
$$
for some $s\ge 1$ and $r\ge 0$. Now we substitute $xt$ for $t_j$ in these identities and obtain the identity $x^syx^{r+1}ty^2\approx yx^4ty^2$. Then, since the identity 
\begin{equation}
\label{xyxty=xxyty}
xyxty\approx x^2yty
\end{equation} 
holds in the variety $\mathbf V$, this variety satisfies
$$
xyxty \stackrel{\eqref{xyx=xyxx}}\approx xyx^{s+r}ty^2 \stackrel{\eqref{xyxty=xxyty}}\approx x^syx^{r+1}ty^2\approx yx^4ty^2\stackrel{\eqref{xyx=xyxx}}\approx yx^2ty,
$$
i.e., the identity
\begin{equation}
\label{xyxty=yxxty}
xyxty\approx yx^2ty.
\end{equation}
Clearly, the identity~\eqref{xzxyty=xzyxty} follows from the identities~\eqref{xyx=xyxx} and~\eqref{xyxty=yxxty}, whence $\mathbf V\subseteq \mathbf I$.

Suppose now that $x\in\simple(\mathbf v'\mathbf p)$. Then Lemma~\ref{w=w'abw''}(ii) and the fact that $\mathbf J$ violates the identity~\eqref{..ayxb..=..axyb..} imply that the first and the second occurrences of $x$ in $\mathbf v$ lie in different blocks in $\mathbf v$, whence $\ell_1(\mathbf v, x)<\ell_1(\mathbf v,t_i)$. Taking into account Corollary~\ref{word problem F vee dual to E}, we have that $\ell_1(\mathbf u, x)<\ell_1(\mathbf u,t_i)$. If the third occurrence of $x$ precedes the first occurrence of $y$ in $\mathbf u$ then $\mathbf V$ satisfies the identities
$$
xt_ixyt_jy\!\stackrel{\eqref{xyx=xyxx}}\approx\! xt_ix^2yt_jy^2=\mathbf u(x,y,t_i,t_j)\approx \mathbf v(x,y,t_i,t_j)=xt_iyx^2t_jy^2\!\stackrel{\eqref{xyx=xyxx}}\approx\! xt_iyxt_jy.
$$
We rename the letters in these identities and obtain that the identity~\eqref{xzxyty=xzyxty} holds in $\mathbf V$. If the third occurrence of $x$ is preceded the first occurrence of $y$ in $\mathbf u$ then $\mathbf V$ satisfies
\begin{align*}
xt_ixyt_jy&\stackrel{\eqref{xyx=xyxx}}\approx xt_ix^2yt_jy^2\stackrel{\eqref{xyxty=xxyty}}\approx xt_ixyxt_jy^2=\mathbf u(x,y,t_i,t_j)\\ 
&\stackrel{\phantom{\eqref{xyx=xyxx}}}\approx\mathbf v(x,y,t_i,t_j)=xt_iyx^2t_jy^2\stackrel{\eqref{xyx=xyxx}}\approx xt_iyxt_jy,
\end{align*}
i.e., the identity~\eqref{xzxyty=xzyxty}.

Finally, suppose that $x\in\con_2(\mathbf v'\mathbf p)$. Then the identities
$$
\mathbf v'\,\mathbf p\mathbf a\,yx\,\mathbf b\,\mathbf v''\stackrel{\eqref{xyxztx=xyxzxtx}}\approx \mathbf v'\,\mathbf p\mathbf a\,xyx\,\mathbf b\,\mathbf v''\stackrel{\eqref{xyxty=xxyty}}\approx \mathbf v'\,\mathbf p\mathbf a\,x^2y\,\mathbf b\,\mathbf v''\stackrel{\eqref{xyx=xyxx}}\approx \mathbf v'\,\mathbf p\mathbf a\,xy\,\mathbf b\,\mathbf v''=\mathbf v
$$ 
hold in the variety $\mathbf J$. We obtain a contradiction with the fact that this variety violates the identity~\eqref{..ayxb..=..axyb..}. So, we have proved that $\mathbf V\subseteq \mathbf I$.

\smallskip

Suppose now that $\mathbf V\subset\mathbf I$. Then we can believe that the identity~\eqref{..ayxb..=..axyb..} does not hold in the variety $\mathbf I$. We will prove that $\mathbf V\subseteq\mathbf H$. Note that the identity~\eqref{xzxyty=xzyxty} allows us to swap a non-latest occurrence and a non-first occurrence of two multiple letters whenever these occurrences are adjacent to each other. Hence $x\notin\con(\mathbf v'\mathbf p)$. Then we can verify that $\mathbf V$ satisfies the identity~\eqref{xyxty=yxxty} (see the second paragraph of Case~1), and therefore, is contained in $\mathbf H$.

\smallskip

Suppose now that $\mathbf V\subset\mathbf H$. Then we can believe that the identity~\eqref{..ayxb..=..axyb..} does not hold in the variety $\mathbf H$. We will prove that $\mathbf V=\mathbf F\vee\overleftarrow{\mathbf E}$. Arguments similar to those from the previous paragraph imply that $x\notin\con(\mathbf v'\mathbf p)$. If the first and the second occurrence of $x$ in $\mathbf v$ lie in the same block then $\mathbf b=\mathbf b'x\mathbf b''$ for some words $\mathbf b'$ and $\mathbf b''$. Then we have:
\begin{align*}
\mathbf v'\mathbf p\mathbf ayx\mathbf b' x\mathbf b''\mathbf v''&{}\approx\mathbf v'\mathbf p\mathbf ayx^2\mathbf b'\mathbf b''\mathbf v''&&\textup{by Lemma~\ref{w=v_1aav_2v_3}}\\
&\approx \mathbf v'\mathbf p\mathbf axyx\mathbf b'\mathbf b''\mathbf v''&&\textup{by the identity~\eqref{xyxty=yxxty}}\\
&\approx \mathbf v'\mathbf p\mathbf ax^2y\mathbf b'\mathbf b''\mathbf v''&&\textup{by the identity~\eqref{xyxty=xxyty}}\\
&\approx \mathbf v'\mathbf p\mathbf axy\mathbf b'x\mathbf b''\mathbf v''.&&\textup{by Lemma~\ref{w=v_1aav_2v_3}}.
\end{align*}
This contradicts the fact that the variety $\mathbf H$ violates the identity~\eqref{..ayxb..=..axyb..}. Therefore, the first and the second occurrence of $x$ in $\mathbf v$ lie in different blocks. Then there is $i<k<j$ such that $\ell_1(\mathbf v, t_k)<\ell_2(\mathbf v,x)$. In view of Corollary~\ref{word problem F vee dual to E}, $\ell_1(\mathbf u, t_k)<\ell_2(\mathbf u,x)$. Then $\mathbf V$ satisfies the identities
$$
xyt_kxt_jy\!\stackrel{\eqref{xyx=xyxx}}\approx\! xyt_kx^2t_jy^2\!=\!\mathbf u(x,y,t_i,t_j)\!\approx\! \mathbf v(x,y,t_i,t_j)\!=\!yxt_kx^2t_jy^2\!\stackrel{\eqref{xyx=xyxx}}\approx\! yxt_kxt_jy.
$$
So, the identity~\eqref{xyzxty=yxzxty} holds in $\mathbf V$. Hence $\mathbf V=\mathbf F\vee\overleftarrow{\mathbf E}$ by Lemma~\ref{basis for F vee dual E}.

\medskip

\emph{Case} 2: the third occurrence of $y$ precedes the second occurrence of $x$ in $\mathbf v$. This case is considered similarly to the previous one. Then there is $j$ such that $\ell_3(\mathbf v, y)<\ell_1(\mathbf v, t_j)<\ell_2(\mathbf v, x)$. In view of Corollary~\ref{word problem F vee dual to E}, $\ell_3(\mathbf u, y)<\ell_1(\mathbf u, t_j)<\ell_2(\mathbf u, x)$.

\smallskip

First, we are going to verify that $\mathbf V\subseteq\mathbf I$. Suppose that $y\notin\con(\mathbf v'\mathbf p\mathbf a)$. Then $\mathbf V$ satisfies the identities 
$$
xy^2t_jx \stackrel{\eqref{xyx=xyxx}}\approx xy^3t_jx^2=\mathbf u(x,y,t_j)\approx \mathbf v(x,y,t_j)=yxy^2tx^2\stackrel{\eqref{xyx=xyxx}}\approx yxytx.
$$
We rename the letters in these identities and obtain that the identity~\eqref{xyxty=yxxty} holds in $\mathbf V$. Clearly, the identity~\eqref{xzxyty=xzyxty} follows from the identity~\eqref{xyxty=yxxty}, whence $\mathbf V\subseteq \mathbf I$.

Suppose now that $y\in\simple(\mathbf v'\mathbf p\mathbf a)$. Then Lemma~\ref{w=w'abw''}(i) and the fact that $\mathbf J$ violates the identity~\eqref{..ayxb..=..axyb..} imply that $x\notin\con(\mathbf v'\mathbf p\mathbf a\mathbf b)$. Besides that, Lemma~\ref{w=w'abw''}(ii) and the fact that $\mathbf J$ violates the identity~\eqref{..ayxb..=..axyb..} imply that the first and the second occurrences of $y$ in $\mathbf v$ lie in different blocks in $\mathbf v$, whence $\ell_1(\mathbf v, y)<\ell_1(\mathbf v,t_i)$. Taking into account Corollary~\ref{word problem F vee dual to E}, we have that $\ell_1(\mathbf u, y)<\ell_1(\mathbf u,t_i)$. Then $\mathbf V$ satisfies the identities
\begin{align*}
yt_ixyt_jx&\stackrel{\eqref{xyx=xyxx}}\approx yt_ixy^2t_jx^2=\mathbf u(x,y,t_i,t_j)\approx\mathbf v(x,y,t_i,t_j)\\
&\stackrel{\phantom{\eqref{xyx=xyxx}}}=yt_iyxyt_jx^2\stackrel{\eqref{xyxty=xxyty}}\approx xt_iy^2xt_jy^2\approx yt_iyxt_jx.
\end{align*}
We rename the letters in these identities and obtain that the identity~\eqref{xzxyty=xzyxty} holds in $\mathbf V$, whence $\mathbf V\subseteq \mathbf I$.

Finally, suppose that $y\in\con_2(\mathbf v'\mathbf p\mathbf a)$. Then the identities
$$
\mathbf v'\,\mathbf p\mathbf a\,yx\,\mathbf b\,\mathbf v''\stackrel{\eqref{xyx=xyxx}}\approx \mathbf v'\,\mathbf p\mathbf a\,y^2x\,\mathbf b\,\mathbf v''\stackrel{\eqref{xyxty=xxyty}}\approx \mathbf v'\,\mathbf p\mathbf a\,yxy\,\mathbf b\,\mathbf v''\stackrel{\eqref{xyxztx=xyxzxtx}}\approx \mathbf v'\,\mathbf p\mathbf a\,xy\,\mathbf b\,\mathbf v''=\mathbf v
$$ 
hold in the variety $\mathbf J$. We obtain a contradiction with the fact that this variety violates the identity~\eqref{..ayxb..=..axyb..}. So, we have proved $\mathbf V\subseteq \mathbf I$.

\smallskip

Suppose now that $\mathbf V\subset\mathbf I$. Then we can believe that the identity~\eqref{..ayxb..=..axyb..} does not hold in the variety $\mathbf I$. We will prove that $\mathbf V\subseteq\mathbf H$. Since the identity~\eqref{xzxyty=xzyxty} allows us to swap a non-latest occurrence and a non-first occurrence of two multiple letters whenever these occurrences are adjacent to each other, we have $y\notin\con(\mathbf v'\mathbf p\mathbf a)$. Then we can verify that $\mathbf V$ satisfies the identity~\eqref{xyxty=yxxty} (see the second paragraph of Case~2), and therefore, is contained in $\mathbf H$.

\smallskip

Suppose now that $\mathbf V\subset\mathbf H$. Then we can believe that the identity~\eqref{..ayxb..=..axyb..} does not hold in the variety $\mathbf H$. We will prove that $\mathbf V=\mathbf F\vee\overleftarrow{\mathbf E}$. Arguments similar to those from the previous paragraph imply that $y\notin\con(\mathbf v'\mathbf p\mathbf a)$. If the first and the second occurrence of $y$ in $\mathbf v$ lie in the same block then $\mathbf b=\mathbf b'y\mathbf b''$ for some words $\mathbf b'$ and $\mathbf b''$. Then we have:
\begin{align*}
\mathbf v'\mathbf p\mathbf ayx\mathbf b' y\mathbf b''\mathbf v''&{}\approx\mathbf v'\mathbf p\mathbf ay^2x\mathbf b'\mathbf b''\mathbf v''&&\textup{by Lemma~\ref{w=v_1aav_2v_3}}\\
&\approx \mathbf v'\mathbf p\mathbf ayxy\mathbf b'\mathbf b''\mathbf v''&&\textup{by the identity~\eqref{xyxty=xxyty}}\\
&\approx \mathbf v'\mathbf p\mathbf axy^2\mathbf b'\mathbf b''\mathbf v''&&\textup{by the identity~\eqref{xyxty=yxxty}}\\
&\approx \mathbf v'\mathbf p\mathbf axy\mathbf b'y\mathbf b''\mathbf v''.&&\textup{by Lemma~\ref{w=v_1aav_2v_3}}.
\end{align*}
This contradicts the fact that the variety $\mathbf H$ violates the identity~\eqref{..ayxb..=..axyb..}. Therefore, the first and the second occurrences of $y$ in $\mathbf v$ lie in different blocks. Then there is $i<k<j$ such that $\ell_1(\mathbf v, t_k)<\ell_2(\mathbf v,y)$. In view of Corollary~\ref{word problem F vee dual to E}, $\ell_1(\mathbf u, t_k)<\ell_2(\mathbf u,y)$. Then $\mathbf V$ satisfies the identities
$$
xyt_kyt_jx\!\stackrel{\eqref{xyx=xyxx}}\approx\! xyt_ky^2t_jx^2\! =\! \mathbf u(x,y,t_i,t_j)\! \approx\!  \mathbf v(x,y,t_i,t_j)\! =\! yxt_ky^2t_jx^2\!\stackrel{\eqref{xyx=xyxx}}\approx\! yxt_kyt_jx.
$$
It follows that the identity~\eqref{xyzxty=yxzxty} holds in $\mathbf V$. Hence $\mathbf V=\mathbf F\vee\overleftarrow{\mathbf E}$ by Lemma~\ref{basis for F vee dual E}.

Proposition~\ref{L(J)} is proved.
\end{proof}

\section{Proof of Theorem~\ref{main result}}
\label{proof of main result}

Put  
\begin{align*}
\mathbf u_n[p,\ell_1,\ell_2,\dots,\ell_n]=xz_1z_2\cdots z_nx^p\, \biggl(\,\prod_{i=1}^n t_iz_i^{\ell_i}\biggr)
\end{align*}
for every positive integers $n$, $p$, $\ell_1,\ell_2,\dots,\ell_n$. We note that $\mathbf u_n[1,1,\dots,1]=\mathbf w_n[\varepsilon]$ where $\varepsilon$ denotes the identity element of $S_n$.

The following statement is evident.

\begin{remark}
\label{subwords of u_n}
Let $\mathbf v=\mathbf u_n[p,\ell_1,\ell_2,\dots,\ell_n]$ for some natural numbers $n$, $p$, $\ell_1,\ell_2,\dots,\ell_n$. If $a,b\in\con(\mathbf v)$ and $a\ne b$ then the subword $ab$ of the word $\mathbf v$ has exactly one occurrence in this word.\qed{\sloppy

}
\end{remark}

\begin{proof}[Proof of Theorem~\ref{main result}.]
It follows from~\cite[Proposition~3.1]{Sapir-87} that the variety $\mathbf J$ is locally finite. Since the variety $\mathbf J$ is small by Proposition~\ref{L(J)}, Lemma~\ref{small+LF=>FG} implies that this variety is finitely generated.

In view of Proposition~\ref{L(J)}, the varieties $\mathbf T$, $\mathbf{SL}$, $\mathbf C$, $\mathbf D$, $\mathbf E$, $\overleftarrow{\mathbf E}$, $\mathbf E\vee\overleftarrow{\mathbf E}$, $\mathbf F$, $\mathbf F\vee\overleftarrow{\mathbf E}$, $\mathbf H$, $\mathbf I$ are the only proper subvarieties of $\mathbf J$. The varieties $\mathbf E\vee\overleftarrow{\mathbf E}$ and $\mathbf F\vee\overleftarrow{\mathbf E}$ are finitely based by Lemmas~\ref{basis for E vee dual E} and~\ref{basis for F vee dual E} respectively. The remaining of the listed varieties are finitely based by their definitions.

We note that to prove that every proper subvariety of $\mathbf J$ is finitely based, one does not need to describe the whole lattice $L(\mathbf J)$. It is sufficient to establish only the fact that $\mathbf I$ is the unique maximal subvariety of $\mathbf J$. Indeed, it is proved in~\cite[Theorem~1.1(i)]{Lee-12} that every monoid variety that satisfies the identities~\eqref{xzytxy=xzytyx} and~\eqref{xyzxty=yxzxty} is finitely based. Obviously, the identities~\eqref{xzytxy=xzytyx} and~\eqref{xyzxty=yxzxty} hold in $\mathbf I$. Taking into account the fact that $\mathbf I$ is the unique maximal subvariety of $\mathbf J$, we obtain that every proper subvariety of $\mathbf J$ is finitely based. We have decided to describe the lattice $L(\mathbf J)$ to prove the fact that the variety $\mathbf J$ is limit, because this description is of certain independent interest and may be useful in further research.

So, it remains to verify that $\mathbf J$ is non-finitely based. Arguing by contradiction, we suppose that $\mathbf J$ has a finite basis of identities $\Sigma$. Let $k$ be a maximum of length of left-hand or right-hand sides of the identities from $\Sigma$. We are going to verify that if $n>k$ then the identity system $\Sigma$ does not imply the identity $\mathbf w_n[\varepsilon]\approx\mathbf w_n'[\varepsilon]$. To establish this fact, it suffices to prove that if an identity $\mathbf u_n[p,k_1,k_2,\dots,k_n]\approx \mathbf w$ follows from the identity system $\Sigma$ then $\mathbf w=\mathbf u_n[q,\ell_1,\ell_2,\dots,\ell_n]$ for some natural numbers $q,\ell_1,\ell_2,\dots,\ell_n$. Put $\mathbf v=\mathbf u_n[p,k_1,k_2,\dots,k_n]$.  By Lemma \ref{deduction} and induction, we can reduce our considerations to the case when $\mathbf v= \mathbf a\xi(\mathbf s)\mathbf b$ and $\mathbf w= \mathbf a\xi(\mathbf t)\mathbf b$ for some $\mathbf a,\mathbf b\in F^1$, $\xi\in\End(F^1)$ and $\mathbf s\approx \mathbf t\in\Sigma$. We can assume without loss of generality that the words $\mathbf v$ and $\mathbf w$ are different.

In view of Corollary~\ref{word problem F vee dual to E}, 
$$
\mathbf w =\mathbf a\xi(\mathbf t)\mathbf b=\mathbf u\, \biggl(\,\prod_{i=1}^n t_iz_i^{\ell_i}\biggr)
$$
for some natural numbers $\ell_1,\ell_2,\dots,\ell_n$ and some word $\mathbf u$ such that $\simple(\mathbf u)=\{z_1,z_2,\dots,z_n\}$ and $\mul(\mathbf u)=\{x\}$. Lemma~\ref{w = ... in H,I,J}(iii) implies that $\mathbf u(z_1,z_2,\dots,z_n)=z_1z_2\cdots z_n$. Besides that, the first occurrence of $x$ in $\mathbf u$ precedes the first occurrence of $z_1$ in $\mathbf u$ by Lemma~\ref{w = ... in H,I,J}(ii). 

If the word $xz_1z_2\cdots z_n$ is a prefix of the word $\mathbf a$ then the required conclusion is evident. Suppose now that $\mathbf a=xz_1z_2\cdots z_i$ for some $0\le i<n$ (if $i=0$ then we mean that $\mathbf a=x$). Then
$$
\xi(\mathbf s)\mathbf b=z_{i+1}z_{i+2}\cdots z_nx^p\, \biggl(\,\prod_{i=1}^n t_iz_i^{k_i}\biggr)\ \text{ and }\ \xi(\mathbf t)\mathbf b=\mathbf u'\, \biggl(\,\prod_{i=1}^n t_iz_i^{\ell_i}\biggr)
$$
for some suffix $\mathbf u'$ of the word $\mathbf u$. Since the identity $\xi(\mathbf s)\mathbf b\approx \xi(\mathbf t)\mathbf b$ holds in $\mathbf J$, the inclusion $\mathbf F\vee\overleftarrow{\mathbf E}\subset \mathbf J$ and Corollary~\ref{word problem F vee dual to E} imply that $\con(\mathbf u')=\{x,z_{i+1},z_{i+2},\dots,z_n\}$. In view of Lemma~\ref{w = ... in H,I,J}(iii), $\mathbf u'(z_{i+1},z_{i+2},\dots,z_n)=z_{i+1}z_{i+2}\cdots z_n$. If $p=1$ then $\mathbf u'(x,z_n)=z_nx$ by Corollary~\ref{word problem F vee dual to E}, and we obtain a contradiction with the inequality $\mathbf v\ne\mathbf w$. If $p>1$ then Lemma~\ref{w = ... in H,I,J}(ii) implies that $\mathbf u'(x,z_n)=z_nx^q$ for some $q>1$, whence $\mathbf w=\mathbf u_n[q,\ell_1,\ell_2,\dots,\ell_n]$.

Finally, suppose that $\mathbf a=\lambda$. If $\xi(\mathbf s)=xz_1z_2\cdots z_j$ for some $0\le j\le n$ then the letters $x,z_1,z_2,\dots, z_j$ are the images of simple letters of the word $\mathbf s$. Then $\xi(\mathbf t)=xz_1z_2\cdots z_j$ by Corollary~\ref{word problem F vee dual to E}. Therefore, we can assume that the word $xz_1z_2\cdots z_nx$ is a prefix of $\xi(\mathbf s)$. 

Let $e_0\mathbf s_0e_1\mathbf s_1\cdots e_m\mathbf s_m$ be the decomposition of the word $\mathbf s$. Then the decomposition of the word $\mathbf t$ has the form $e_0\mathbf t_0e_1\mathbf t_1\cdots e_m\mathbf t_m$ by the inclusion $\mathbf E\subset\mathbf J$ and Corollary~\ref{decompositions of u and v in E}. Remark~\ref{subwords of u_n} and the fact that the length of the word $\mathbf s$ is less than $n$ imply that there is $j\in\{1,2,\dots n\}$ such that $z_j\in\con(\xi(e_i))$ for some $i\in\{1,2,\dots,m\}$. We can assume that $j$ is the greatest number with this property. Let $j'\le j$ be the least number with $z_{j'}\in\con(\xi(e_i))$. Then $xz_1z_2\cdots z_{j'-1}=\xi(e_0\mathbf s_0e_1\mathbf s_1\cdots e_{i-1}\mathbf s_{i-1})$ and the word $z_{j'}z_{j'+1}\cdots z_j$ is a prefix of $\xi(e_i)$. In view of Corollary~\ref{word problem F vee dual to E}, we have that $xz_1z_2\cdots z_{j'-1}=\xi(e_0\mathbf t_0e_1\mathbf t_1\cdots e_{i-1}\mathbf t_{i-1})$. If $j=n$ then we obtain the required conclusion. So, we can assume that $j<n$. Then $z_n\in\con(\xi(a))$ for some $a\in\mul(\mathbf s)$. It follows that there is $r$ such that $i<r$, $\ell_1(\mathbf s, e_r)<\ell_2(\mathbf s, a)$ and $t_1\in\con(\xi(e_r))$. In view of Remark~\ref{subwords of u_n}, $\xi(a)=z_n$. 

Suppose that $\ell_2(\mathbf u,x)<\ell_1(\mathbf u,z_n)$. Clearly, $\ell_1(\mathbf u,z_j)<\ell_2(\mathbf u,x)$. Then there is $b\in\mul(\mathbf t)=\mul(\mathbf s)$ such that $x\in\con(\xi(b))$. Remark~\ref{subwords of u_n} implies that $\xi(b)=x^h$ for some $h\in\mathbb N$. Clearly, $b$ does not occur in the words $\mathbf s_r,\mathbf t_r,\mathbf s_{r+1},\mathbf t_{r+1},\dots, \mathbf s_m,\mathbf t_m$. If $\ell_1(\mathbf s,a)<\ell_1(\mathbf s,b)$ then $\mathbf J$ satisfies the identities
$$
ab^{f_1}e_ra^{f_2}=\mathbf s(a,b,e_r)\approx \mathbf t(a,b,e_r)=b^{g_1}ab^{g_2}e_ra^{g_3}
$$
for some $f_2,g_1,g_3\in \mathbb N$, $f_1\ge 2$ and $g_2\ge 0$. This contradicts Lemma~\ref{w = ... in H,I,J}(ii). If $\ell_1(\mathbf s,b)<\ell_1(\mathbf s,a)$ then $b\in\simple(\mathbf s_0\mathbf s_1\cdots \mathbf s_{i-1})$. Taking into account Corollary~\ref{word problem F vee dual to E}, we obtain that $b\in\simple(\mathbf t_0\mathbf t_1\cdots \mathbf t_{i-1})$. Then $\mathbf J$ satisfies the identities
$$
be_iab^{f_1}e_ra^{f_2}=\mathbf s(a,b,e_i,e_r)\approx \mathbf t(a,b,e_i,e_r)=be_ib^{g_1}ab^{g_2}e_ra^{g_3}
$$
for some $f_1,f_2,g_1,g_3\in \mathbb N$ and $g_2\ge 0$. A contradiction with Lemma~\ref{w = ... in H,I,J}(i). Therefore, $\ell_1(\mathbf u,z_n)<\ell_2(\mathbf u,x)$. This implies that $\mathbf u=xz_1z_2\cdots z_n x^q$ for some $q$.
\end{proof}

\begin{proof}[Proof of Corollary~\ref{J is Cross}.]
The variety $\mathbf J$ is non-Cross by Theorem~\ref{main result}. Let $\mathbf X$ be a proper subvariety of $\mathbf J$. In view of Theorem~\ref{main result}, $\mathbf X$ is finitely based. According to Proposition~\ref{L(J)}, $\mathbf X$ is small. Then $\mathbf X$ is finitely generated by Lemma~\ref{small+LF=>FG}. So, $\mathbf X$ is a Cross variety.
\end{proof}

\subsection*{Acknowledgments.} The author is sincerely grateful to Professor Boris Vernikov for his attention and assistance in the writing of the article, to Dr. Edmond W.H. Lee for offering  of a shorter proof of Theorem~\ref{main result} and his numerous valuable suggestions for improving the manuscript and to the anonymous referee
for several useful remarks.


\small

\begin{thebibliography}{99}
\bibitem{Edmunds-77}
C.C.Edmunds, \emph{On certain finitely based varieties of semigroups}, Semigroup Forum, \textbf{15} (1977), 21--39.
\bibitem{Gusev-Vernikov-18}
S.V.Gusev and B.M.Vernikov, \emph{Chain varieties of monoids}, Dissertationes Math., \textbf{534} (2018), 1--73.
\bibitem{Jackson-05}
M.Jackson, \emph{Finiteness properties of varieties and the restriction to finite algebras}, Semigroup Forum, \textbf{70} (2005), 154--187; \emph{Erratum to ``Finiteness properties of varieties and the restriction to finite algebras''}, Semigroup Forum, \textbf{96} (2018), 197--198.
\bibitem{Jackson-Lee-18}
M.Jackson and E.W.H.Lee, \emph{Monoid varieties with extreme properties}, Trans. Amer. Math. Soc., \textbf{370} (2018), 4785--4812.
\bibitem{Kozhevnikov-12}
P.A.Kozhevnikov, \emph{On nonfinitely based varieties of groups of large prime exponent}, Commun. Algebra, \textbf{40} (2012), 2628--2644.
\bibitem{Lee-08}
E.W.H.Lee, \emph{On the variety generated by some monoid of order five}, Acta Sci. Math. (Szeged), \textbf{74} (2008), 509--537.
\bibitem{Lee-09}
E.W.H.Lee, \emph{Finitely generated limit varieties of aperiodic monoids with central idempotents}, J. of Algebra and its Applications, \textbf{40} (2009), 779--796.
\bibitem{Lee-12}
E.W.H.Lee, \emph{Maximal Specht varieties of monoids}, Moscow Math. J., \textbf{12} (2012), 787--802.
\bibitem{Lee-13}
E.W.H.Lee, \emph{Almost Cross varieties of aperiodic monoids with central idempotents}, Beitr\"age zur Algebra und Geometrie, \textbf{54} (2013), 121--129.
\bibitem{Sapir-87}
M.V.Sapir, \emph{Problems of Burnside type and the finite basis property in varieties of semigroups}, Izv. Akad. Nauk SSSR Ser. Mat. \textbf{51}, No. 2 (1987), 319--340 [Russian; Engl. translation: Math. USSR-Izv. \textbf{30}, No. 2 (1988), 295--314].
\bibitem{Shevrin-Volkov-85}
L.N.Shevrin and M.V.Volkov, \emph{Identities of semigroups}, Izv. VUZ. Matematika, No. 11 (1985), 3--47 [Russian; Engl. translation: Soviet Math Izv. VUZ,  \textbf{29}, No. 11 (1985), 1--64].
\bibitem{Volkov-01}
M.V.Volkov, \emph{The finite basis problem for finite semigroups}, Sci. Math. Jpn., \textbf{53} (2001), 171--199.
\bibitem{Wismath-86}
S.L.Wismath, \emph{The lattice of varieties and pseudovarieties of band monoids}, Semigroup Forum, \textbf{33} (1986), 187--198.
\bibitem{Zhang-13}
W.T.Zhang, \emph{Existence of a new limit variety of aperiodic monoids}, Semigroup Forum, \textbf{86} (2013), 212--220.
\bibitem{Zhang-Luo-19+}
W.T.Zhang and Y.F.Luo, \emph{A new example of limit variety of aperiodic monoids},  to appear; available at: https://arxiv.org/abs/1901.02207.
\end{thebibliography}
\end{document}